\documentclass[a4paper,12pt]{amsart}

\usepackage[ansinew]{inputenc}
\usepackage[english]{babel} 
\usepackage[T1]{fontenc}
\usepackage{csquotes}
\usepackage{amsmath,amssymb,amsthm, mathtools, bbm, ifxetex, amsrefs}
\usepackage{trfsigns}
\usepackage{xspace}
\usepackage{geometry}
\usepackage{comment}
\usepackage{enumerate}
\usepackage{hyperref}

\numberwithin{equation}{section}

\newtheorem{satz}{Theorem}[section]
\newtheorem{proposition}[satz]{Proposition}
\newtheorem{lemma}[satz]{Lemma}
\newtheorem{korollar}[satz]{Corollary}

\theoremstyle{definition}

\title{Higher Order Concentration in presence of Poincaré-type inequalities}

\author{Friedrich G\"{o}tze}
\address{Friedrich G\"{o}tze, Faculty of Mathematics, Bielefeld University, Bielefeld, Germany}
\email{goetze@math.uni-bielefeld.de}
\author{Holger Sambale}
\address{Holger Sambale, Faculty of Mathematics, Bielefeld University, Bielefeld, Germany}
\email{hsambale@math.uni-bielefeld.de}

\begin{document}

\thanks{This research was supported by CRC 1283.}
\begin{abstract}
	We show sharpened forms of the concentration of measure phenomenon typically centered at stochastic expansions of order $d-1$ for any $d \in \mathbb{N}$.
	Here we focus on differentiable functions on the Euclidean space in presence
	of a Poincaré-type inequality. The bounds are based on $d$-th order derivatives.
\end{abstract}
\date{\today}

\maketitle

\section{Introduction}

In this note, we study higher order versions of the concentration of 
measure phenomenon. Instead of the
classical problem of deviations of $f$ around the mean 
$\mathbb{E}f$, we study potentially smaller fluctuations of 
$\tilde{f}_d := f - \mathbb{E}f - f_1 - \ldots - f_d$, where $f_1, \ldots, f_d$ are 
``lower order terms'' of $f$ with respect to a suitable decomposition,
such as a Taylor-type decomposition of $f$. In order to study the
concentration of $\tilde{f}_d$ around $0$, which we call higher order
concentration of measure, we use derivatives up to order $d$.

 
Previous work includes R. Adamczak and 
P. Wolff \cite{A-W}, who exploited certain Sobolev-type inequalities or subgaussian
tail conditions to derive exponential tail inequalities for
functions with bounded higher-order derivatives (evaluated in terms of some 
tensor-product matrix norms). This approach was continued by R. Adamczak, W. Bednorz and P. Wolff for measures satisfying modified logarithmic Sobolev inequalities in \cite{A-B-W}. While in \cite{A-W}, concentration around 
the mean is studied, the idea of sharpening concentration inequalities for 
Gaussian and related measures by requiring orthogonality to linear functions 
also appears in P. Wolff \cite{W} as well as in
D. Cordero-Erausquin, M. Fradelizi and B. Maurey \cite{CE-F-M}. For a detailed overview of the concentration of measure phenomenon, see \cite{L, B-L-M}.

Our research started with second order results for functions on 
the $n$-sphere orthogonal to linear functions \cite{B-C-G},
with an approach which has been extended in \cite{G-S} for measures satisfying
logarithmic Sobolev inequalities. This includes discrete models as well as 
differentiable functions on open subsets of $\mathbb{R}^n$.
These results were extended to arbitrary higher orders in \cite{B-G-S}.

While in \cite{B-G-S}, measures satisfying a logarithmic Sobolev
inequality were considered, the aim of this note is to prove similar results for
measures satisfying a Poincaré-type inequality, i.\,e. a weaker assumption.
To this end, let us recall that a Borel probability measure $\mu$
on an open set $G \subset \mathbb{R}^n$ is said to satisfy 
a \emph{Poincaré-type inequality} with constant $\sigma^2 > 0$
if for any bounded smooth function $f$ on $G$ with gradient $\nabla f$,
\begin{gather}
\label{PI1}
\text{Var}_\mu (f) \le \sigma^2 \int |\nabla f|^2\, d\mu.
\end{gather}
Here, $\text{Var}_\mu (f) = \int f^2\, d\mu - (\int f\, d\mu)^2$ denotes the
variance. When considering $\sigma$ instead of $\sigma^2$ itself, we
will always assume it to be positive.

Given a function $f \in \mathcal{C}^d(G)$,
we define $f^{(d)}$ to be the (hyper-) matrix whose entries
\begin{equation}
\label{Hesseallgem}
f^{(d)}_{i_1 \ldots i_d}(x) = \partial_{i_1 \ldots i_d} f(x), \qquad
d = 1,2,\dots
\end{equation}
represent the $d$-fold (continuous) partial derivatives of $f$ at $x \in G$. 
By considering $f^{(d)}(x)$ as a symmetric multilinear $d$-form, we define 
operator-type norms by
\begin{equation}
\label{Operatornorm}
|f^{(d)}(x)|_\mathrm{Op} = 
\sup \left\{ f^{(d)}(x)[v_1, \ldots, v_d] \colon |v_1| = \ldots |v_d| = 1\right\}.
\end{equation}
For instance, $|f^{(1)}(x)|_\mathrm{Op}$ is the Euclidean norm of 
the gradient $\nabla f(x)$, and $|f^{(2)}(x)|_\mathrm{Op}$ is the operator norm 
of the Hessian $f''(x)$. Furthermore, we will use 
the short-hand notation
\begin{equation}
\label{Lpnorm2}
\lVert f^{(d)} \rVert_{\mathrm{Op}, p} =
\left(\int_G |f^{(d)}|_\mathrm{Op}^p\, d\mu\right)^{1/p}, \qquad p \in (0, \infty].
\end{equation}
For $p = \infty$, the right-hand side has to be read as the $L^\infty$-norm of $|f^{(d)}|_\mathrm{Op}$.

We now have the following:

\begin{satz}
\label{kontinuierlich}
Let $\mu$ be a probability measure on an open set $G \subset \mathbb{R}^n$ satisfying a Poincaré-type inequality
with constant $\sigma^2 > 0$, and let 
$f \colon G \to \mathbb{R}$ be a $\mathcal{C}^d$-smooth function with 
$\int_G f\, d\mu = 0$. Assuming the conditions
\begin{align}
\label{Bed1}
\lVert f^{(k)} \rVert_{\mathrm{Op},2} &\le 
\sigma^{d-k}\qquad \forall k = 1, \ldots, d-1,\\
\label{Bed2}
\lVert f^{(d)} \rVert_{\mathrm{Op}, \infty} &\le 1,
\end{align}
there exists some universal constant $c > 0$ such that
$$
\int_G \exp \left(\frac{c}{\sigma}\, |f|^{1/d}\right) d\mu \, \le \, 2.
$$
\end{satz}

Here, a possible choice is $c = 1/(12\e)$. Comparing Theorem \ref{kontinuierlich}
to its analogue in presence of a logarithmic Sobolev inequality, i.\,e. Theorem 1.6
in \cite{B-G-S}, we see that
under the same assumptions \eqref{Bed1} and \eqref{Bed2}, logarithmic
Sobolev inequalities yield exponential moment bounds for $|f|^{2/d}$,
whereas Poincaré-type inequalities provide exponential moments
for $|f|^{1/d}$ only. This corresponds to the well-known behaviour in case of $d=1$.

If $f$ has centered partial derivatives of order up to $d-1$, 
it is possible to replace \eqref{Bed1} by a somewhat simpler condition. 
To this end, we need to involve 
Hilbert--Schmidt-type norms $|f^{(d)}(x)|_\mathrm{HS}$ 
defined as the Euclidean 
norm of $f^{(d)}(x) \in \mathbb{R}^{n^d}$. Similarly to \eqref{Lpnorm2}, 
$\lVert f^{(d)}\rVert_{\mathrm{HS}, 2}$ then denotes the $L^2$-norm
of $|f^{(d)}|_\mathrm{HS}$. In detail:

\begin{satz}
\label{kontinuierlichmAbl}
Let $\mu$ be a probability measure on an open set $G \subset \mathbb{R}^n$ satisfying a Poincaré-type
inequality with constant $\sigma^2$, and let 
$f \colon G \to \mathbb{R}$ be a $\mathcal{C}^d$-smooth function such that
$$
\int_G f\, d\mu = 0\qquad \text{and}\qquad 
\int_G \partial_{i_1 \ldots i_k} f\, d\mu = 0
$$
for all $k = 1, \ldots, d-1$ and $1 \le i_1, \ldots, i_k \le n$.
Assuming that
$$
\lVert f^{(d)} \rVert_{\mathrm{HS}, 2} \le 1\qquad \text{and}\qquad 
\lVert f^{(d)} \rVert_{\mathrm{Op}, \infty} \le 1,
$$
there exists some universal constant $c > 0$ such that
$$
\int_G \exp \left(\frac{c}{\sigma}\, |f|^{1/d}\right) d\mu \le 2.
$$

\end{satz}

Here again, a possible choice is $c = 1/(12\e)$.

By Chebyshev's inequality, Theorem \ref{kontinuierlich} immediately
yields
$$\mu (|f| \ge t) \le 2 \e^{-ct^{1/d}/\sigma}$$
for any $t \ge 0$. For small values of $t$, it is possible to obtain
refined tail estimates in the spirit of R. Adamczak \cite{A}, Theorem 7,
or R. Adamczak and P. Wolff \cite{A-W}, Theorem 3.3 (with $\gamma = 1$ using their notation), by analyzing the
proof of Theorem~\ref{kontinuierlich}:

\begin{korollar}
	\label{KorrTails}
Let $\mu$ be a probability measure on an open set $G \subset \mathbb{R}^n$ satisfying a Poincaré-type
inequality with constant $\sigma^2 > 0$, and let 
$f \colon G \to \mathbb{R}$ be a $\mathcal{C}^d$-smooth function with 
$\int_G f\, d\mu = 0$. For any $t \ge 0$, set
$$\eta_f(t) := \min \Big(\frac{\sqrt{2}t^{1/d}}{\sigma \lVert f^{(d)} \rVert_{\mathrm{Op}, \infty}^{1/d}}, \min_{k=1, \ldots, d-1}
\frac{\sqrt{2}t^{1/k}}{\sigma \lVert f^{(k)} \rVert_{\mathrm{Op},2}^{1/k}} \Big).$$
Then,
$$\mu(|f| \ge t) \le \e^2 \exp(-\eta_f(t)/(d\e)).$$
\end{korollar}

As a generalization of these bounds, we may consider measures
satisfying weighted Poincaré-type inequalities. Recall that
a Borel probability measure $\mu$ on an open set $G \subset \mathbb{R}^n$
is said to satisfy  a \emph{weighted Poincaré-type inequality}
if for any bounded smooth function $f$ on $G$ with gradient $\nabla f$,
\begin{gather}
\label{wPI}
\text{Var}_\mu (f) \le  \int |\nabla f|^2w^2\, d\mu,
\end{gather}
where $w \colon G \to [0, \infty)$ is some measurable function. Examples
include Cauchy measures and Beta distributions. For a detailed discussion see
S.\,G. Bobkov and M. Ledoux \cite{B-L}.

In these cases we cannot expect exponential integrability as in Theorem
\ref{kontinuierlich} any more, since distributions satisfying
\eqref{wPI} may have a slow, say, polynomial, decay at infinity.
Nevertheless, it is still possible to obtain higher order concentration
results by controlling the $L^p$-norms of $f$ and its derivatives.
In detail:

\begin{proposition}
	\label{wPItails}
Let $\mu$ be a probability measure on an open set $G \subset \mathbb{R}^n$ satisfying a weighted Poincaré-type inequality \eqref{wPI}, and let 
$f \colon G \to \mathbb{R}$ be a $\mathcal{C}^d$-smooth function with
$\int_G f d\mu= 0$.
Then, for any $p \ge 2$,
\begin{align*}
\lVert f \rVert_p
& \ \le \
\sum_{k=1}^{d-1} (2^{\frac{k-2}{2}}p \lVert w \rVert_{2^kp})^k\, \lVert f^{(k)} 
\rVert_{\mathrm{Op}, 2}	+
(2^{\frac{d-2}{2}}p)^d \lVert w \rVert_{2^{d-1}p}^{d-1}\, \lVert w \lvert f^{(d)}\rvert_{\mathrm{Op}} \rVert_{2^{d-1}p}\\
& \ \le \
\sum_{k=1}^{d-1} (2^{\frac{k-2}{2}}p \lVert w \rVert_{2^kp})^k\, \lVert f^{(k)} 
\rVert_{\mathrm{Op}, 2}	+
(2^{\frac{d-2}{2}}p \lVert w \rVert_{2^dp})^d\, \lVert f^{(d)} \rVert_{\mathrm{Op}, 2^dp}.
\end{align*}
\end{proposition}

Proposition \ref{wPItails} should be compared to \eqref{pfstep} from
the proof of Theorem \ref{kontinuierlich} in Section 2. In particular,
if the weight function $w$ is bounded by some real number
$\sigma > 0$, $\mu$ clearly satisfies a Poincaré-type inequality
\eqref{PI1} with constant $\sigma^2$. In this case, Proposition
\ref{wPItails} implies a slightly weaker version of \eqref{pfstep},
and it is possible to derive Theorem \ref{kontinuierlich} again
though with a somewhat weaker constant $c = c_d > 0$.

Suitable conditions on the weight function $w$ may still yield exponential-type tails at
least in certain intervals. For instance, the following higher order analogue of
Corollary 4.2 in \cite{B-L} holds:

\begin{korollar}
\label{CorrExpTails}
Let $\mu$ be a probability measure on an open set $G \subset \mathbb{R}^n$ satisfying a weighted Poincaré-type
inequality \eqref{wPI}, and let  $f \colon G \to \mathbb{R}$ be a
$\mathcal{C}^d$-smooth function with $\int_G f d\mu= 0$ and such that \eqref{Bed1} (with $\sigma^2 = 1$)
and \eqref{Bed2} from Theorem \ref{kontinuierlich} hold. Assume $\lVert w
\rVert_{2^dp} \le C$ for some $p \ge 2$ and some $C \ge 2^{-(d-1)/2}$.
Then, for any $0 \le t \le (2^{\frac{d+5}{2}}C \e p)^d$,
$$\mu(|f| \ge t) \le \e^{d/\e} \exp(-dt^{1/d}/(2^{\frac{d+5}{2}}C\e)).$$
\end{korollar}

Hence, we obtain exponential-type tail bounds on an interval of length proportional
to $p^d$. Note that if $t > (2^{\frac{d+5}{2}}C \e p)^d$, we may still give bounds on $\mu(|f| \ge t)$ by taking \eqref{genbd} for $q = p$ from the proof of Corollary \ref{CorrExpTails}. We omit details at this point.
The assumption $C \ge 2^{-(d-1)/2}$ is needed for technical reasons. In fact, it guarantees that the quantities $(2^{\frac{k-1}{2}} C)^k$, $k \le d-1$, are bounded by $(2^{\frac{d-1}{2}} C)^d$. For $d = 1$ it can be removed. It is possible to adapt the proof for $0 < C < 2^{-(d-1)/2}$ and obtain similar bounds.

For $d=1$, Corollary \ref{CorrExpTails} gives back a version of Corollary 4.2 from \cite{B-L} up to constants, though with a boundedness condition on $\lVert w \rVert_{2p}$ rather than $\lVert w \rVert_p$. This may be adjusted by working with the first inequality from Proposition \ref{wPItails}, in which case we directly get back the \cite{B-L} result. In the same way, it is possible to derive a result similar to Corollary \ref{CorrExpTails} which requires bounds on $\lVert w \rVert_{2^{d-1}p}$. We have chosen to work with the second inequality from Proposition \ref{wPItails} instead (and thus need bounds on $\lVert w \rVert_{2^dp}$) since this is technically slightly more convenient.

Under stronger moment conditions on the weight function $w$, e.\,g.
$\int \e^{w^2/\alpha} d\mu \le 2$ for some $\alpha > 0$, it is possible to obtain
exponential-type tail bounds even on the whole positive half-line, cf. Corollary 4.3
in \cite{B-L}.

\subsection*{Outline}
In Section 2, we give the proofs of the results stated above. In
Section 3, we provide some applications, including homogeneous
multilinear polynomials of order $d$ and linear eigenvalue statistics
in random matrix theory.

\section{Proofs}

Given a continuous function on an open subset $G \subset \mathbb{R}^n$, 
the equality
\begin{equation}
\label{generalizedmodulus}
|\nabla f(x)| = \limsup_{x \to y} \frac{|f(x) - f(y)|}{|x-y|}, \qquad x \in G,
\end{equation}
may be used as definition of the generalized modulus of the gradient
of $f$. The function $|\nabla f|$ is Borel measurable, and if $f$ is 
differentiable at $x$, the generalized modulus of 
the gradient agrees with the Euclidean norm of the usual gradient.
This operator 
preserves many identities from calculus in form of inequalities, such as 
a ``chain rule inequality''
\begin{equation}
\label{chainrule}
|\nabla T(f)| \le |T'(f)||\nabla f|,
\end{equation}
where $|T'|$ is understood according to \eqref{generalizedmodulus} again.

As shown in \cite{B-G-S}, Lemma 4.1, using the generalized modulus of
the gradient, the operator norms of the derivatives of 
consecutive orders are related as follows:

\begin{lemma}
\label{itGradHess}
Given a $\mathcal{C}^d$-smooth function 
$f \colon G \to \mathbb{R}$, $d \in \mathbb{N}$, at all points $x \in G$,
$$
|\nabla \lvert f^{(d-1)}(x) \rvert_\mathrm{Op}| \le 
\lvert f^{(d)}(x) \rvert_\mathrm{Op}.
$$
\end{lemma}

\begin{proof}
Indeed, for any $h \in \mathbb{R}^n$, by the triangle inequality,
\begin{align*}
 & 
\big|\,\lvert f^{(d-1)}(x+h) \rvert_\mathrm{Op} - 
\lvert f^{(d-1)}(x) \rvert_\mathrm{Op}\big|
 \, \le \,
\lvert f^{(d-1)}(x+h) - f^{(d-1)}(x) \rvert_\mathrm{Op}\\
 = \, &
\sup \{ (f^{(d-1)}(x+h) - f^{(d-1)}(x))[v_1, \ldots, v_{d-1}] 
\colon v_1, \ldots, v_{d-1} \in S^{n-1} \},
\end{align*}
while, by the Taylor expansion,
$$
(f^{(d-1)}(x+h) - f^{(d-1)}(x))[v_1, \ldots, v_{d-1}] = 
f^{(d)}(x)[v_1, \ldots, v_{d-1}, h] + o(|h|)
$$
as $h \to 0$. Here, the $o$-term can be bounded by a quantity which is 
independent of $v_1, \ldots, v_{d-1} \in S^{n-1}$. As a consequence,
\begin{align*}
 &
\limsup_{h \to 0} 
\frac{|\,\lvert f^{(d-1)}(x+h) \rvert_\mathrm{Op} - \lvert f^{(d-1)}(x) 
\rvert_\mathrm{Op}|}{|h|}\\
 \le \; &
\sup \{ f^{(d)}(x)[v_1, \ldots, v_{d-1}, v_d] \colon v_1, \ldots, v_d \in S^{n-1}\}
 = \lvert f^{(d)}(x) \rvert_\mathrm{Op}.
\end{align*}
\end{proof}

Following the scheme of proof developed in \cite{B-G-S}, we moreover need to establish a recursion for 
the $L^p$-norms of the derivatives 
of $f$ of consecutive orders. To this end, 
we recall a classical result on the moments of Lipschitz functions in the
presence of Poincaré-type inequalities. Here, similarly to \eqref{Lpnorm2}, we write
$$\lVert \nabla g \rVert_{\mathrm{Op}, p} =
\left(\int_G |\nabla g|^p\, d\mu\right)^{1/p}, \qquad p \in (0, \infty],$$
for any locally Lipschitz function $g$ on $G$ with generalized modulus of gradient $|\nabla g|$. In detail:

\begin{lemma}
\label{momentslemma}
Let $\mu$ be a probability measure on an open set $G \subset \mathbb{R}^n$ satisfying a Poincaré-type inequality with
constant $\sigma^2 > 0$, and let $g \colon G \to \mathbb{R}$ be locally Lipschitz with
$\int_G g d\mu = 0$. Then, for any $p \ge 2$,
\begin{equation}
\label{momentscentered}
\int_{G} \lvert g \rvert^p d\mu \le \Big(\frac{\sigma p}{\sqrt{2}}\Big)^p
\int_G \lvert \nabla g \rvert^p d\mu.
\end{equation}
In particular, for any $g \colon G \to \mathbb{R}$ locally Lipschitz,
\begin{equation}\label{moments}
\lVert g \rVert_p \, \le \, \lVert g \rVert_2 +
\frac{\sigma p}{\sqrt{2}}\, \lVert \nabla g \rVert_p.
\end{equation}
\end{lemma}

Note that in \eqref{moments}, $g$ is not required to have mean $0$.
For the reader's convenience, let us briefly recall the proof.

\begin{proof}
By standard arguments, we may assume $g$ to be $\mathcal{C}^1$-smooth and bounded. Moreover,
by the subadditivity property of the variance functional, the Poincaré-type
inequality for the probability measure $\mu$ on $G$ is extended to the same relation
on $G \times G$, i.\,e.
\begin{equation}
\label{produktpoincare}
\mathrm{Var}_{\mu^2} (u) \le \sigma^2 \iint \lvert \nabla u(x,y) \rvert^2 d\mu(x)d\mu(y)
\end{equation}
for the product measure $\mu^2 = \mu \otimes \mu$. Here, for any
$\mathcal{C}^1$-smooth function $u = u(x,y)$, the modulus of the gradient is given by
$$\lvert \nabla u(x,y) \rvert^2 = \lvert \nabla_x u(x,y) \rvert^2 +
\lvert \nabla_y u(x,y) \rvert^2.$$

Now consider the function
$$u(x,y) = \lvert g(x) - g(y) \rvert^{\frac{p}{2}} \mathrm{sign} (g(x)-g(y)),$$
which is $\mathcal{C}^1$-smooth for $p > 2$ with modulus of gradient
$$\lvert \nabla u(x,y) \rvert = \frac{p}{2} \lvert g(x) - g(y) \rvert^{\frac{p}{2}-1} \sqrt{\lvert \nabla g(x) \rvert^2 + \lvert \nabla g(y) \rvert^2}.$$
Since $u$ has a symmetric distribution under $\mu^2$, applying
\eqref{produktpoincare} together with Hölder's inequality yields
\begin{align*}
&\frac{1}{\sigma^2}\iint \lvert g(x) - g(y) \rvert^p d \mu^2(x,y)\\
&\le \frac{p^2}{4} \iint \lvert g(x) - g(y) \rvert^{p-2}
\big(\lvert \nabla g(x) \rvert^2 + \lvert \nabla g(y) \rvert^2 \big)
d\mu^2(x,y)\\
&\le \frac{p^2}{4} \Big(\iint \lvert g(x) - g(y) \rvert^p d\mu^2(x,y) \Big)^{\frac{p-2}{p}}
\Big(\iint \big(\lvert \nabla g(x) \rvert^2 + \lvert \nabla g(y) \rvert^2 \big)^{\frac{p}{2}} d\mu^2(x,y)\Big)^{\frac{2}{p}}.
\end{align*}
By Jensen's inequality, the last integral may be bounded by
$$2^{\frac{p}{2}-1} \iint (\lvert \nabla g(x) \rvert^p + \lvert \nabla g(y) \rvert^p)d\mu^2(x,y) = 2^{\frac{p}{2}} \int \lvert \nabla g \rvert^p d\mu.$$
Consequently,
$$\Big(\iint \lvert g(x) - g(y) \rvert^p d\mu^2(x,y) \Big)^{\frac{2}{p}} \le \frac{\sigma^2 p^2}{2} \Big(\int \lvert \nabla g \rvert^p d\mu\Big)^{\frac{2}{p}},$$
or, equivalently,
$$\iint \lvert g(x) - g(y) \rvert^p d\mu^2(x,y) \le \Big(\frac{\sigma p}{\sqrt{2}}\Big)^p \int \lvert \nabla g \rvert^p d\mu.$$
In particular, the latter inequality shows that any locally Lipschitz function $g$ such that the right-hand side is finite is integrable (if $g$ is unbounded, we may perform a simple truncation argument).
If $\int g d\mu = 0$, it follows from Jensen's
inequality that the left integral can be bounded below by $\int |g|^p d\mu$, which proves \eqref{momentscentered}.
To see \eqref{moments}, it remains to note that by the triangle inequality,
\begin{equation*}
\Big\lVert g - \int g d\mu \Big\rVert_p \, \ge \,
\lVert g \rVert_p - \Big\lvert\int g d\mu\Big\rvert \, \ge \,
\lVert g \rVert_p - \lVert g \rVert_2.
\end{equation*}
\end{proof}

Combining Lemma \ref{itGradHess} and \eqref{moments}, we are able 
to prove Theorem \ref{kontinuierlich}. Recall that if a relation 
of the form
\begin{equation}
\label{subexp}
\lVert f \rVert_k \le \gamma k \qquad (k \in \mathbb{N})
\end{equation}
holds true with some constant $\gamma > 0$, then $f$ has sub-exponential 
tails, i.\,e. $\int \e^{c|f|} d\mu \linebreak[3]\le 2$ for some constant $c = c(\gamma) > 0$,
e.\,g. $c = \frac{1}{2 \gamma \e}$. Indeed, using $k! \ge (\frac{k}{\e})^k$,
we have
$$
\int \exp(c|f|) d \mu = 1 + \sum_{k=1}^{\infty} c^k \frac{\int |f|^k d\mu}{k!} 
\le 1 + \sum_{k=1}^{\infty} (c \gamma)^k \frac{k^k}{k!}
\le 1 + \sum_{k=1}^{\infty} (c \gamma \e)^k = 2.
$$

\begin{proof}[Proof of Theorem \ref{kontinuierlich}]
Using \eqref{moments} with $f$ replaced by 
$\lvert f^{(k-1)} \rvert_\mathrm{Op}$, $2 \le k \leq d$, we get
\begin{align}
\label{MomenteiteriertdiffA}
\begin{split}
\lVert f^{(k-1)} \rVert_{\mathrm{Op}, p}
 & \le 
\lVert f^{(k-1)} \rVert_{\mathrm{Op}, 2} + 
\frac{\sigma p}{\sqrt{2}}\, \lVert \nabla \lvert f^{(k-1)} \rvert_\mathrm{Op} \rVert_p\\
 & \le 
\lVert f^{(k-1)} \rVert_{\mathrm{Op}, 2} + 
\frac{\sigma p}{\sqrt{2}}\, \lVert f^{(k)} \rVert_{\mathrm{Op}, p},
\end{split}
\end{align}
where  Lemma \ref{itGradHess} was applied on the last step. Consequently, using
\eqref{momentscentered} and then \eqref{MomenteiteriertdiffA} iteratively,
\begin{equation}
\label{pfstep}
\lVert f \rVert_p \
 \le \ \sum_{k=1}^{d-1} \Big(\frac{\sigma p}{\sqrt{2}}\Big)^k\, \lVert f^{(k)} 
\rVert_{\mathrm{Op}, 2}	+ 
\Big(\frac{\sigma p}{\sqrt{2}}\Big)^d\, \lVert f^{(d)} \rVert_{\mathrm{Op}, p}.
\end{equation}
Since 
$\lVert f^{(k)} \rVert_{\mathrm{Op}, 2} \le \sigma^{d-k}$ 
for all $k = 1, \ldots, d-1$ and 
$\lVert f^{(d)} \rVert_{\mathrm{Op}, \infty} \le 1$
by assumption, we obtain
\begin{equation}
\label{MomenteiteriertdiffB}
\lVert f \rVert_p \, \le \, \sigma^{d} \sum_{k=1}^{d} (p/\sqrt{2})^k
 \, \le \,
\frac{1}{1 - (p/\sqrt{2})^{-1}}\, (\sigma p/\sqrt{2})^d \, \le \, 4\, (\sigma p/\sqrt{2})^d
\end{equation}
and therefore
$
\lVert f \rVert_p \le (3 \sigma p)^{d}
$
for all $p \ge 2$. Moreover,
$
\lVert f \rVert_p \le \lVert f \rVert_2 \le (6 \sigma)^d
$
for $p < 2$. It follows that
$$
\lVert \lvert f \rvert^{1/d} \rVert_k = \lVert f \rVert_{k/d}^{1/d} \le \gamma k
$$
for all $k \in \mathbb{N}$, i.\,e. \eqref{subexp} holds with
$\gamma = 6 \sigma$ (and $\lvert f \rvert^{1/d}$ in place of $f$). This yields the assertion of the theorem.
\end{proof}

\begin{proof}[Proof of Theorem \ref{kontinuierlichmAbl}]
Starting as in the proof of Theorem \ref{kontinuierlich}, we arrive at
\begin{align}
\lVert f \rVert_p 
 & \le 
\sum_{k=1}^{d-1} (\sigma p/\sqrt{2})^k\, \lVert f^{(k)} \rVert_{\mathrm{HS}, 2} + 
(\sigma p/\sqrt{2})^d\, \lVert f^{(d)} \rVert_{\mathrm{Op}, p},\label{MomenteiteriertA}
\end{align}
where we used that operator norms are dominated by Hilbert--Schmidt norms.
Moreover, since 
$\int_G \partial_{i_1 \ldots i_k} f\, d\mu = 0$,
by the Poincaré-type inequality,
$$
\int_G (\partial_{i_1 \ldots i_k} f)^2\, d\mu \le 
\sigma^2 \sum_{j=1}^{n} \int_G (\partial_{i_1 \ldots i_k j} f)^2\, d\mu
$$
whenever $1 \le i_1, \ldots, i_k \le n$, $k \leq d-1$. Summing over all 
$1 \le i_1, \ldots, i_k \le n$, we get
\begin{equation}
\label{MomenteiteriertB}
\lVert f^{(k)} \rVert_{\mathrm{HS}, 2}^2 = 
\int_G \lvert f^{(k)} \rvert_\mathrm{HS}^2\, d\mu \le 
\sigma^2 \int_G \lvert f^{(k+1)} \rvert_\mathrm{HS}^2\, d\mu = 
\sigma^2\, \lVert f^{(k+1)} \rVert_{\mathrm{HS}, 2}^2.
\end{equation}
Using \eqref{MomenteiteriertB} in \eqref{MomenteiteriertA} and iterating, 
we thus obtain
\begin{equation*}
\lVert f \rVert_p \le 
\sum_{k=1}^{d-1} \sigma^{d} (p/ \sqrt{2})^k\, \lVert f^{(d)} 
\rVert_{\mathrm{HS}, 2} + 
(\sigma p/ \sqrt{2})^d\, \lVert f^{(d)} \rVert_{\mathrm{Op}, p}.
\end{equation*}
Noting that $\lVert f^{(d)} \rVert_{\mathrm{HS}, 2} \le 1$ and 
$\lVert f^{(d)} \rVert_{\mathrm{Op}, \infty} \le 1$, we arrive at 
\eqref{MomenteiteriertdiffB},
from where we may proceed as in the proof of Theorem \ref{kontinuierlich}.
\end{proof}

\begin{proof}[Proof of Corollary \ref{KorrTails}]
First note that by Chebyshev's inequality, for any $p \ge 1$
\begin{equation}\label{tailsexp}
\mu(|f| \ge \e \lVert f \rVert_p) \le \e^{-p}.
\end{equation}
Moreover, if $p \ge 2$, it follows from \eqref{pfstep} that
$$\e\lVert f \rVert_p  \le \ \e\Big(\sum_{k=1}^{d-1} (\sigma p/\sqrt{2})^k\, \lVert f^{(k)} 
\rVert_{\mathrm{Op}, 2} + 
(\sigma p/\sqrt{2})^d\, \lVert f^{(d)} \rVert_{\mathrm{Op}, \infty}\Big).$$
Assuming $\eta_f(t) \ge 2$, we therefore arrive at
$$\e\lVert f \rVert_{\eta_f(t)} \le \e\big(\sum_{k=1}^{d-1} t + t\big)
= (d\e) t.$$
Hence, applying \eqref{tailsexp} to $p = \eta_f(t)$ (if $p \ge 2$) yields
$$\mu(|f| \ge (d\e) t) \le \mu(|f| \ge \e \lVert f \rVert_{\eta_f(t)})
\le \exp(-\eta_f(t)).$$
Using a trivial estimate provided that $p = \eta_f(t) < 2$, we obtain
$$\mu(|f| \ge (d\e) t) \le \e^2 \exp(-\eta_f(t))$$
for all $t \ge 0$.
The proof now easily follows by rescaling $f$ by $d\e$ and
using that $\eta_{d\e f}(t) \ge \eta_f(t)/(d\e)$.
\end{proof}

In order to prove Proposition \ref{wPItails}, we have to adapt the
first steps of the proof of Theorem \ref{kontinuierlich}. First, we
have the following generalization of Lemma \ref{momentslemma} (in fact, this is a
version of Theorem 4.1 in \cite{B-L}):

\begin{lemma}
	\label{momentslemmawPI}
	Let $\mu$ be a probability measure on an open set $G \subset \mathbb{R}^n$ satisfying a weighted Poincaré-type inequality \eqref{wPI},
	and let $g \colon G \to \mathbb{R}$ be locally Lipschitz with
	$\int_G g d\mu = 0$. Then, for any $p \ge 2$,
	\begin{equation}
	\label{momentscenteredwPI}
	\int_{G} \lvert g \rvert^p d\mu \le \Big(\frac{p}{\sqrt{2}}\Big)^p
	\int_G \lvert \nabla g \rvert^p w^p \,d\mu.
	\end{equation}
	In particular, for any $g \colon G \to \mathbb{R}$ locally Lipschitz,
	\begin{equation}\label{momentswPI}
	\lVert g \rVert_p \, \le \, \lVert g \rVert_2 +
	\frac{p}{\sqrt{2}}\, \lVert w \lvert\nabla g\rvert \rVert_p.
	\end{equation}
\end{lemma}

The proof of Lemma \ref{momentslemmawPI} uses similar arguments as the
proof of Lemma \ref{momentslemma}, and we therefore omit it.
In particular, by Hölder's inequality, \eqref{momentswPI} implies
	\begin{equation}\label{momentswPIH}
\lVert g \rVert_p \, \le \, \lVert g \rVert_2 +
\frac{p}{\sqrt{2}}\, \lVert w \rVert_{2p} \lVert \nabla g \rVert_{2p}.
\end{equation}
Starting with \eqref{momentscenteredwPI}--\eqref{momentswPIH} and iterating as in
\eqref{MomenteiteriertdiffA} and \eqref{pfstep}, we obtain
$$\lVert f \rVert_p
\ \le \
\sum_{k=1}^{d-1} 2^{\binom{k}{2}} \Big(\frac{p \lVert w \rVert_{2^kp}}{\sqrt{2}}\Big)^k\, \lVert f^{(k)} 
\rVert_{\mathrm{Op}, 2}	+ 2^{\binom{d}{2}}
\Big(\frac{p \lVert w \rVert_{2^{d-1}p}}{\sqrt{2}}\Big)^d\, \lVert w \lvert f^{(d)}\rvert_{\mathrm{Op}} \rVert_{2^{d-1}p},$$
hence we easily arrive at the conclusions of Proposition \ref{wPItails}. Again, we
omit the details.

Finally, the proof of Corollary \ref{CorrExpTails} is similar to the proof of
Corollary 4.2 in \cite{B-L}.

\begin{proof}[Proof of Corollary \ref{CorrExpTails}]
First let $2 \le q \le p$. Using the assumptions and Proposition \ref{wPItails}, we arrive at
$$\lVert f \rVert_q \, \le \, \sum_{k=1}^{d-1} (2^{\frac{k-2}{2}}q C)^k +
(2^{\frac{d-2}{2}}q C)^d$$
and hence
$$\lVert f \rVert_q \le 4 \, (2^{\frac{d-1}{2}} Cq)^d \le
(2^{\frac{d+3}{2}} Cq)^d$$
(this follows as in \eqref{MomenteiteriertdiffB}, substituting $\sigma$ by
$2^{\frac{d-1}{2}}C \ge 1$). Moreover, if $0 < q \le 2$, we have
$$\lVert f \rVert_q \le \lVert f \rVert_2 \le (2^{\frac{d+5}{2}} C)^d.$$
Since the function $q \mapsto \e^{d/\e} q^{dq}$, $q > 0$, is minimized at $q = 1/\e$
with minimum value $1$, it follows that $\mathbb{E}|f|^q \le \e^{d/\e}\,
(2^{\frac{d+5}{2}} C q)^{dq}$ for all $0 < q \le p$. Therefore, for any $t > 0$ and
any $0 < q \le p$,
\begin{equation}\label{genbd}
\mu(|f| \ge t) \le \frac{\mathbb{E}|f|^q}{t^q} \le
\e^{d/\e}\, \bigg(\frac{(2^{\frac{d+5}{2}} C q)^d}{t}\bigg)^q.
\end{equation}
Now set $s = t^{1/d}/(2^{\frac{d+5}{2}} C)$ and write $\mu(|f| \ge t) \le \e^{d/\e}
\e^{-\varphi(q)}$ with $\varphi(q) = dq(\log(s) - \log(q))$. It is easy to check
that $\varphi$ is a concave function on $(0,\infty)$ which attains its maximum at
$q_0 = s/\e$ with $\varphi(q_0) = ds/\e = dt^{1/d}/(2^{\frac{d+5}{2}} C\e)$. Noting
that $q_0 \le p$ is equivalent with $t \le (2^{\frac{d+5}{2}} C \e p)^d$ completes
the proof.
\end{proof}

\section{Applications}

Let $X_1, \ldots, X_n$ be independent random variables with
distributions satisfying a Poincaré-type inequality \eqref{PI1} with
common constant $\sigma^2 > 0$.
For real numbers $a_{i_1 \ldots i_d}$, $i_1 < \ldots < i_d$, consider
the function
\begin{equation}\label{multPol}
f(X_1, \ldots, X_n) := \sum_{i_1 < \ldots < i_d} a_{i_1 \ldots i_d} X_{i_1} \cdots X_{i_d},
\end{equation}
which is a homogeneous multilinear polynomial of order $d$. For any $i_1 < \ldots <
i_d$ and any permutation $\sigma \in S^d$, set $a_{\sigma(i_1) \ldots \sigma(i_d)}
\equiv a_{i_1 \ldots i_d}$. Moreover, set $a_{i_1 \ldots i_d} = 0$ whenever the
indexes $i_1, \ldots, i_d$ are not pairwise different. This gives rise to a
hypermatrix $A = (a_{i_1 \ldots i_d}) \in \mathbb{R}^{n^d}$, whose Euclidean norm
we denote by $\lVert A \rVert_{\mathrm{HS}}$. Moreover, set $\lVert A \rVert_\infty
:= \max_{i_1 < \ldots < i_d} \lvert a_{i_1 \ldots i_d}\rvert$.

As a first example, we may apply our results to functions of type
\eqref{multPol}. Here it is convenient to assume for the random variables $X_i$ to
have mean zero:

\begin{proposition}
\label{multPolProp}
Let $X_1, \ldots, X_n$ be independent random variables with distributions satisfying
a Poincaré-type inequality \eqref{PI1} with common constant $\sigma^2 > 0$. Assume
$\mathbb{E} X_i = 0$ for all $i = 1, \ldots, n$. Let $d \in
\mathbb{N}$, and consider a function $f$ of type \eqref{multPol}. Then,
$$\mathbb{E} \exp \Big(\frac{c}{\sigma \lVert A \rVert_{\mathrm{HS}}^{1/d}} \lvert f
\rvert^{1/d} \Big) \le 2.$$
Here, $\mathbb{E}$ denotes the expectation with respect to the random variables
$X_1, \ldots,\linebreak[2] X_n$, and $c$ is the absolute constant appearing in Theorem
\ref{kontinuierlichmAbl}. In particular,
$$\mathbb{E} \exp \Big(\frac{c}{\sigma n^{1/2} \lVert A \rVert_{\infty}^{1/d}}
\lvert f \rvert^{1/d} \Big) \le 2.$$
Moreover, if $\mathbb{E}X_i^2 = 1$ for all $i = 1, \ldots, n$,
\begin{align*}
\mathbb{P} (\lvert f - \mathbb{E} f\rvert \ge t) &\le \e^2 \exp \Big(-
\frac{\sqrt{2}}{\sigma d\e}
\min\Big(\frac{t}{\lVert A \rVert_\mathrm{HS}}, \frac{t^{1/d}}
{\lVert A \rVert_\mathrm{HS}^{1/d}}\Big)\Big)\\
&\le \e^2 \exp \Big(- \frac{\sqrt{2}}{\sigma d\e}\min\Big(\frac{t}{n^{d/2}
\lVert A\rVert_\infty}, \frac{t^{1/d}}{n^{1/2}\lVert A \rVert_\infty^{1/d}}\Big)\Big).
\end{align*}
\end{proposition}

Proposition \ref{multPolProp} follows immediately from Theorem
\ref{kontinuierlichmAbl} and Corollary \ref{KorrTails}.
Note that for non-centered random variables $X_1, \ldots, X_n$, applying Proposition
\ref{multPolProp} to the random variables $X_i - \mathbb{E} X_i$ means removing
certain ``lower order'' terms in \eqref{multPol}, which is in accordance with the
ideas sketched in the introduction.

We may furthermore apply our results in the context of random matrix theory. Here we
extend an example on second order concentration bounds for linear eigenvalue
statistics in presence of a logarithmic Sobolev inequality \cite{G-S}, Proposition
1.10, to the situation where only a Poincaré-type inequality is available.

Indeed, let $\{\xi_{jk}, 1 \le j \le k \le N \}$ be a family of independent random
variables on some probability space. Assume that the distributions of the
$\xi_{jk}$'s all satisfy a (one-dimensional) Poincaré-type inequality \eqref{PI1}
with common constant $\sigma^2$. Put $\xi_{jk} = \xi_{kj}$ for $1 \le k < j \le N$
and consider a symmetric $N \times N$ random matrix $\Xi =
(\xi_{jk}/\sqrt{N})_{1 \le j, k \le N}$ and denote by $\mu^{(N)}$ the joint
distribution of its ordered eigenvalues $\lambda_1 \le \ldots \le \lambda_N$ on
$\mathbb{R}^N$ (in fact, $\lambda_1 < \ldots < \lambda_N$ a.s.). Recall that by a
simple argument using the Hoffman--Wielandt theorem, $\mu^{(N)}$ satisfies a
Poincaré-type inequality with constant
\begin{equation}
\label{PC}
\sigma_N^2 = \frac{2 \sigma^2}{N}
\end{equation}
(see for instance S.\,G. Bobkov and F. G\"{o}tze \cite{B-G}). Note that similar
observations also hold for Hermitian random matrices.

Considering the probability space $(\mathbb{R}^N, \mathbb{B}^N, \mu^{(N)})$, if
$f \colon \mathbb{R} \to \mathbb{R}$ is a $\mathcal{C}^1$-smooth function, it is
well-known that asymptotic normality
\begin{equation}
\label{S_N beta}
S_N = \sum_{j=1}^{N} (f(\lambda_j) - \mathbb{E}f(\lambda_j)) \Rightarrow \mathcal{N}(0, \sigma_f^2)
\end{equation}
holds for the self-normalized linear eigenvalue statistics $S_N$. Here,
``$\Rightarrow$'' denotes weak convergence, $\mathbb{E}$ means taking the expectation
with respect to $\mu^{(N)}$ and $\mathcal{N}(0, \sigma_f^2)$ denotes a normal
distribution with mean zero and variance $\sigma_f^2$ depending on $f$. This result
was established by K. Johansson \cite{J} for the case of $\beta$-ensembles and, for
general Wigner matrices, by A.\,M. Khorunzhy, B.\,A. Khoruzhenko and L.\,A. Pastur
\cite{K-K-P} as well as Ya. Sinai and A. Soshnikov \cite{S-S}. Concentration of
measure results have been studied by A. Guionnet and O. Zeitouni \cite{G-Z}, in
particular proving fluctuations of order $\mathcal{O}_\mathbb{P}(1)$. Our results
yield a second order concentration bound:

\begin{proposition}
\label{secondorderWigner}
Let $\mu^{(N)}$ be the joint distribution of the ordered eigenvalues of $\Xi$. Let
$f \colon \mathbb{R} \to \mathbb{R}$ be a $\mathcal{C}^2$-smooth function with
$f'(\lambda_j) \in L^1(\mu^{(N)})$ and bounded second derivatives, and let
$$\tilde{S}_N := S_N - \sum_{j=1}^{N} (\lambda_j - \mathbb{E}(\lambda_j)) \mathbb{E}
f'(\lambda_j)$$
with $S_N$ as in \eqref{S_N beta}. Then, we have
$$\mathbb{E} \exp\Big(\frac{cN^{1/4}}{\sqrt{2}\sigma\lVert f''\rVert_\infty^{1/2}}
\lvert\tilde{S}_N\rvert^{1/2}\Big) \le 2,$$
where $c > 0$ is the absolute constant from Theorem \ref{kontinuierlichmAbl}.
\end{proposition}

Since $\tilde{S}_N$ is ``centered'' in the sense of Theorem \ref{kontinuierlichmAbl},
Proposition \ref{secondorderWigner} immediately follows from elementary calculus,
using \eqref{PC}. Note that in view of the self-normalizing property of $S_N$, the
fluctuation result for $\tilde{S}_N$ is of the next order, although the scaling is of
order $\sqrt{N}$ only. Comparing Proposition \ref{secondorderWigner} to \cite{G-S},
Proposition 1.10, we see that we essentially arrive at the same result though for
$|\tilde{S}_N|^{1/2}$ instead of $|\tilde{S}_N|$ due to the assumption of a
Poincaré-type inequality.

Using Corollary \ref{KorrTails}, we can in fact slightly sharpen the results on the
tail behavior of $S_N$. Indeed, an easy calculation yields
$$\mu_N(\lvert S_N\rvert \ge t)
\le \ \e^2 \exp \Big(- \frac{1}{\sigma d\e}
\min \Big(\frac{tN^{1/2}}{(\int \sum_i (f'(\lambda_i))^2 d\mu_N)^{1/2}},
\frac{t^{1/2}N^{1/4}}{\lVert f''\rVert_\infty^{1/2}} \Big) \Big)$$
for any $t \ge 0$. Similar results may be obtained for higher orders $d \ge 3$.

%
%



\begin{thebibliography}{99}	
	
	\bibitem[A]{A} Adamczak, R.: 
	\textit{Moment inequalities for $U$-statistics}. 
	The Annals of Probability 34(6) (2006), 2288--2314.
	
	\bibitem[A-B-W]{A-B-W} Adamczak, R., Bednorz, W., Wolff, P.:
	\textit{Moment estimates implied by modified log-Sobolev inequalities}.
	ESAIM Probab. Stat. 21 (2017), 467--494.

\bibitem[A-W]{A-W} Adamczak, R., Wolff, P.: 
\textit{Concentration inequalities for non-Lipschitz functions with bounded 
derivatives of higher order}. 
Probability Theory Related Fields 162(3) (2015), 531--586.

\bibitem[B-C-G]{B-C-G} Bobkov, S.\,G., Chistyakov, G.\,P., Götze, F.: 
\textit{Second Order Concentration on the Sphere}. 
Commun. Contemp. Math. 19(5) (2017), 1650058, 20 pp.

\bibitem[B-G]{B-G} Bobkov, S.\,G., G\"{o}tze, F.: \textit{Concentration of empirical distribution functions with applications to non-i.i.d. models}. Bernoulli 16(4) (2010), 1385--1414.

\bibitem[B-G-S]{B-G-S} Bobkov, S.\,G., Götze, F., Sambale, H.: 
\textit{Higher order concentration of measure}. 
Commun. Contemp. Math 21(3) (2019), 1850043, 36 pp.

\bibitem[B-L]{B-L} Bobkov, S.\,G., Ledoux, M.:
\textit{Weighted Poincar\'{e}-type inequalities for Cauchy and other convex measures}.
The Annals of Probability 37(2) (2009), 403--427.

\bibitem[B-L-M]{B-L-M} Boucheron, S., Lugosi, G., Massart, P.: 
\textit{Concentration Inequalities. A Nonasymptotic Theory of Independence}. 
Oxford University Press, 2013.

\bibitem[CE-F-M]{CE-F-M} Cordero-Erausquin, D., Fradelizi, M., Maurey, B.: 
\textit{The (B) conjecture for the Gaussian measure of dilates of symmetric 
convex sets and related problems}. J. Funct. Anal. 214(2) (2004), 410--427.

\bibitem[G-S]{G-S} Götze, F., Sambale, H.: 
\textit{Second order concentration via logarithmic Sobolev inequalities}. 
Preprint, arXiv:1605.08635, to appear in: Bernoulli.

\bibitem[G-Z]{G-Z} Guionnet, A., Zeitouni, O.: \textit{Concentration of the spectral measure for large matrices}. Electron. Commun. Prob. 5 (2000), 119--136.

\bibitem[J]{J} Johansson, K.: \textit{On fluctuations of eigenvalues of random Hermitian matrices}. Duke Math. J. 91(1) (1998), 151--204.

\bibitem[K-K-P]{K-K-P} Khorunzhy, A.\,M., Khoruzhenko, B.\,A., Pastur, L.\,A.: \textit{Asymptotic properties of large random matrices with independent entries}. J. Math. Phys. 37 (1996), 5033-5060.

\bibitem[L]{L} Ledoux, M.: 
\textit{The Concentration of Measure Phenomenon}. 
American Mathematical Society, 2001.

\bibitem[S-S]{S-S} Sinai, Ya., Soshnikov, A.: \textit{A central limit theorem for traces of large random matrices with independent matrix elements}. Bol. Soc. Brasil. Mat (N.S.) 29 (1998), 1--24.

\bibitem[W]{W} Wolff, P.: 
\textit{On some Gaussian concentration inequality for non-Lipschitz functions}. High Dimensional Probability VI, Progress in Probability 66 (2013), 103--110.



\end{thebibliography}
\end{document}